\documentclass[12pt,twoside,reqno]{amsart}
\usepackage[colorlinks=true,citecolor=blue]{hyperref}
\usepackage{mathptmx, amsmath, amssymb, amsfonts, amsthm, mathptmx, enumerate, color}
\usepackage{graphicx}
\usepackage{epstopdf}

\def\R{\mathbb R}

\def\gr{{\rm gr }}
\def\int{{\rm int\,}}

\def\zer{{\rm zer}}

\newtheorem{theorem}{Theorem}[section]

\newtheorem{lemma}{Lemma}[section]

\theoremstyle{definition}
\newtheorem{definition}{Definition}[section]

\newtheorem{remark}{Remark}[section]

\newtheorem{assumption}{Assumption}[section]
\numberwithin{equation}{section}
\newtheorem{algorithm}{Algorithm}
\begin{document}
	\setcounter{page}{1}

	\vspace*{1.0cm}
	\title[FRABA with double inertial effects for non-monotone inclusion problems]
	{A forward-reflected-anchored-backward splitting algorithm with double inertial effects for solving non-monotone inclusion problems}
	\author[N.V. Tran]{Nam V. Tran$^{1,*}$}
	\maketitle
	\vspace*{-0.6cm}

	\begin{center}
		{\footnotesize  {\it							
				$^1$Faculty of Applied Sciences, HCMC University of Technology and Education, Ho Chi Minh City, Vietnam; }}

	\end{center}

	\vskip 4mm {\small\noindent {\bf Abstract.}
In this paper, we study inclusion problems where the involved operators may not be monotone in the classical sense. Specifically, we assume the operators to be generalized monotone, a weaker notion than classical monotonicity. This allows us to extend the applicability of our results to a broader class of operators. We apply the two-step inertial forward-reflected-anchored-backward splitting algorithm proposed in \cite{CHIN} to these non-monotone inclusion problems. We establish the strong convergence of the sequence generated by the algorithm and demonstrate its applicability to other optimization problems, including Constrained Optimization Problems, Mixed Variational Inequalities, and Variational Inequalities.

	
	\noindent {\bf [MSC Classification]} {47J20; 47J30;49J40; 49J53; 49M37} }

	\renewcommand{\thefootnote}{}
	\footnotetext{ $^*$Corresponding author.
		\par
		E-mail addresses: namtv@hcmute.edu.vn (Nam V. Tran).
		\par
	}
	\section{Introduction} 
	Let $H$ be a real Hilbert space with the inner product denoted by $\langle \cdot,\cdot \rangle$ and the norm denoted by $\|\cdot\|$. 
	In this paper, we study the following inclusion problem: Find \(u^*\in K\) such that 
	\begin{equation}\label{inc1}
		0\in F(u^*) + G(u^*)
	\end{equation}
	where $F: H \longrightarrow 2^{H}$ is a set-valued mapping and $G:H \longrightarrow H$ is a singleton mapping, $K$ is a nonempty, closed subset of $H$.  In this paper, we will denote by $\mbox{\rm zer } F$  the \emph{set of zero points} of an operator \(F\).
    
    This problem serves as a broad mathematical model that encompasses numerous well-known problems, including optimization problems, variational inequalities, saddle point problems, Nash equilibrium problems in noncooperative games, and fixed point problems. Many of these can be reformulated as special cases of \eqref{inc1}; see, for instance, \cite{BlumOettli94, Cav} and references therein.
    
	
	For example,  consider a variational inequality problem of finding  \(u^*\in K\) such that $\langle T(u^*), u-u^*\rangle \geq 0,$ for all $ u\in K$, where \(K\) is a nonempty, closed, convex subset of $H$ and $T: H \longrightarrow H$. This problem can be rewritten as the inclusion problem $0\in F(u)$ where 
	\begin{equation*}
		F(x)=\begin{cases} 
			T(u)+N_K(u) &\mbox{ if } u\in K\\
			\emptyset &\mbox{ if } x\neq  K,
		\end{cases}
	\end{equation*} 
	with  $N_K(u)$  the normal cone to $K$ at $u$.
	
	In addition, a constrained minimization problem can also be formulated as an inclusion problem. Indeed, given a proper  convex  function $g: H \longrightarrow (-\infty, +\infty]$  and $K$ be a nonempty closed, convex subset of $H$, then $u^*\in K$ is a solution to the constrained optimization problem (COP): $\min_{u\in K} g(u)$ only if it is a solution to the inclusion problem 
	$0\in  \partial g (u)+N_K(u)$, (see, for instance, \cite{ROCK}).
	
	Many methods have been developed to solve inclusion problems of the form \eqref{inc1}, see, for instance, \cite{BauschkeCombettes}. The forward-backward splitting method, introduced in seminal works by Passty \cite{PAS} and Lions and Mercier \cite{MER}, has become a standard approach for such problems \cite{ABBAS, BauschkeCombettes, BSV, CSE, MAL, PARS, TSENG}. For more details, the following forward-backward splitting method was introduced in \cite{MER}.
    \begin{equation*}
        u_{k+1}=(id+\gamma F)^{-1}(u_k-\gamma Gu_k), \quad k\geq 1.
    \end{equation*}
    Weak convergence was established under strict conditions on $G$ while achieving strong convergence required even more restrictive assumptions \cite{TAK, WANG2018}.
	
To address these limitations, Tseng \cite{TSENG} introduced an improved algorithm in 2000:
\begin{equation*}
    \left\{\begin{array}{lll}
        v_k&= & (id+\gamma F)^{-1}(u_k-\gamma Bu_k),\\
        u_{k+1}&= & v_k-\gamma Gv_k+\gamma Gu_k.
    \end{array}\right. 
\end{equation*}
However, this method incurs additional computational costs due to requiring two forward evaluations of $G$. 
To overcome this disadvantage in \cite{MAL} Malitsky and Tam employed a reflection technique to propose the following update:
\begin{equation*}
    u_{k+1}=(id+\gamma F)^{-1}(u_k-2\gamma G u_k+\gamma Gu_{k-1}), \quad  \gamma \in \left(0, \frac{1}{2L} \right). 
\end{equation*}

Originating from the discretization of the heavy ball method, the inertial technique has become popular in algorithm studies due to its role in accelerating convergence. For example, in \cite{BING}, Bing and Cho introduced the following one-step inertial viscosity-type forward-backward-forward splitting algorithm:
\begin{equation*}
   \left\{ \begin{array}{ccc}
        t_k&=&u_k+\theta_k(u_k-u_{k-1}  \\
        v_k&=&(id+\gamma F)^{-1}(t_k-\gamma Gt_k),\\  
        w_k&=&v_k-\gamma G v_k+\gamma Gu_k,\\
        u_{k+1}&=&\alpha_kfu_k+(1-\alpha_k)w_k.
    \end{array}\right. 
\end{equation*}
Strong convergence was attained for monotone operators.

More recently, two-step inertial techniques have been successfully incorporated into various algorithms \cite{CHIN, Liang}, demonstrating significant improvements in convergence rates \cite{Liang}.
	
In addition, up to now, classical assumptions of monotonicity in inclusion problems have been deeply ingrained \cite{BC_SICON, 22, BCJMAA}. Relaxing these conditions is challenging since fundamental results may no longer hold. For instance, if $(G+F)$ lacks strong monotonicity, the inclusion $0\in (G+F)(x)$ may have no solution. Consequently, the number of algorithms for non-monotone inclusion problems is very limited. Moreover, the monotonicity assumption may restrict the applicability of the results, as operators in real-world applications are often not monotone. Hence, reducing this assumption is a crucial aspect of developing algorithms for these inclusion problems, which serves as the motivation for this research.

In this paper, we extend the concept of monotonicity by allowing a generalized monotonicity framework, where operators may have a negative modulus of monotonicity. Additionally, operators are not required to be maximal. This broader perspective enables the study of a wider class of operators beyond the traditional monotone setting.

The remainder of the paper is structured as follows:

	Section \ref{Preliminaries} revisits fundamental definitions and concepts, and presents several technical lemmas. Specifically, we provide some characterizations for an operator to be maximal generalized monotone and for the sum of two generalized monotone operators to be maximal.

    Section \ref{sec3} presents the main results, including an analysis of the strong convergence of the two-step inertial forward-reflected-anchored-backward (FRAB) splitting algorithm.
Section \ref{sec5} discusses some applications of the algorithm to  COPs, MVIs, and VIs.
Section \ref{sec6} concludes with final remarks.

	\section{Preliminaries} \label{Preliminaries}
	
In this section, we review essential definitions that will be useful in the subsequent discussion.

	\subsection{Some notions on convex analysis} 
	
	Let $f: H  \to (-\infty, +\infty]$ be a convex and lower semicontinuous (l.s.c.) function. Its  \emph{domain} is defined as $\mbox{\rm dom }f=\{x\in H: f(x)<+\infty\}$, and  $f$ is said to be \emph{proper} if $\mbox{\rm{dom}} f\neq \emptyset.$

	A proper, convex and lower semicontinuous function  $f: H  \to (-\infty, +\infty]$ is   \emph{subdifferentiable} at $u$ if its \emph{subdifferential}  at  $u$,  given by $$
	\partial f(u)= \{ w \in H: \, f(v) -  f(u) \geq  \left\langle w, v - u \right\rangle  \, \forall v \in H \}  
	$$
    is non-empty. 
Any  $w \in \partial f(u)$ is called a \emph{subgradient} of $f$ at $u$.

For a nonempty, closed, convex subset $K$ of $H$, the \emph{normal cone} at  $u\in K$, $N_K(u)$,  is defined as
	$$ N_K(u)=\left\{w\in H:\left\langle w, u-v\right\rangle \ge 0, \forall v\in K\right\}, $$
	and $N_K(u) = \emptyset$ if $u \not\in K$. Recall that The \emph{indicator function} of $K$, 
	$$i_K(u) = \begin{cases}
		0 & \mbox{ if } u\in K\\
		+\infty & \mbox{  otherwise} 
	\end{cases}$$ satisfies 	$\partial i_K(u) = N_K(u)$ for all $u\in H$.\\
A fundamental tool in inclusion problems is the \emph{metric projection}, defined for any $u\in H$ as,
	$$
	P_K (u)=\arg\min\left\{\left\|v-u\right\|:v\in K\right\}.
	$$
	Note that when $K$ is nonempty, closed, and convex, $P_K (u)$ exists and is unique. 

We now review some useful identities needed for the convergence analysis in the sequel.

\begin{lemma} \cite{CHIN}\label{simple}
	Let $x,y,z \in \R^n$ and $a,b,\beta \in \mathbb{R}$. Then
	\begin{itemize}
		\item[(a)] We have that 	\begin{eqnarray}\label{eq2.2} 
			&&\|(1+a)x-(a-b)y-bz\|^2 \notag \\
			&=& (1+a)\|x\|^2-(a-b)\|y\|^2-b\|z\|^2 +(1+a)(a-b)\|x-y\|^2 \notag \\
			&&+b(1+a)\|x-z\|^2-b(a-b)\|y-z\|^2.
		\end{eqnarray}
		\item[(b)] The following identity holds
        \begin{eqnarray}\label{eq2.1}
			\langle x - z, y - x\rangle = \frac{1}{2}\|z - y\|^2 - \frac{1}{2}\|x - z\|^2 - \frac{1}{2}\|y-x\|^2.
		\end{eqnarray}
		\item[(c)]  It holds	that \begin{eqnarray*}
			\|\beta x + (1-\beta)y\|^2 = \beta\|x\|^2 + (1-\beta)\|y\|^2 - \beta(1-\beta)\|x - y\|^2.
		\end{eqnarray*}
	\end{itemize}
\end{lemma}

Next, we recall some results related to the convergence properties of sequences, which will be applied in the convergence analysis

\begin{lemma}\cite{SAE}\label{lm2.1} 
    Let $\{s_k\}$  be a sequence  with $s_k\geq 0$ for all $k$, $\{\lambda_k\}$ be a real sequence with $\lambda_k\in (0,1)$ for all $k$ such that $\sum\limits_{k=1}^\infty \lambda_k=\infty$, and  $\{a_k\}$ be a real sequence satisfying 
    \begin{equation*}
        s_{k+1}\leq (1-\lambda_k)s_k+\lambda_ka_k, \quad \forall k\geq 1.
    \end{equation*}
    Assume further that $\limsup\limits_{i\to\infty} a_{k_i}\leq 0$ for each subsequence $\{a_{k_i}\}$ of $\{a_k\}$ satisfying $\liminf\limits_{i\to\infty} (a_{k_i+1}-a_{k_i})\geq 0.$ Then $\lim\limits_{k\to\infty} a_k=0.$
\end{lemma}
\begin{lemma}\cite{MAI} \label{lm2.3} 
    Let $\{s_k\}$ and $\{a_k\}$ be two nonnegative real sequences, $\{\lambda_k\}$ be a sequence in $(0,1)$, and $\{b_k\}$ be a real sequence satisfying 
    \begin{equation*}
        s_{k+1}\leq (1-\lambda_k)s_k+a_k+b_k, \quad \forall k\geq 1, 
    \end{equation*}
    and $\sum\limits_{k=1}^\infty b_k<\infty$, $a_k\leq \lambda_k C$ for some $C\geq 0.$ Then $\{a_k\}$ is bounded. 
    
\end{lemma}

For further details, see \cite{BauschkeCombettes}.

\subsection{Monotone operators}

In this subsection, we review some notions related to operators, especially the definition of monotonicity.  Let $F: H \longrightarrow 2^{H}$ be a set-valued mapping on $H$. 	The \emph{graph} of $F$ is defined as   $\mbox{\rm gr}(F)=\{(x, u)\in H\times H: u\in F(x)\}$. The \emph{domain} and \emph{range} of $F$ are given by 
$$\mbox{\rm dom } F=\{u\in H: F(u)\neq \emptyset\}$$ and 
$$\mbox{\rm ran } F=\{y\in H: \mbox{ there exists } x\in H, y\in F(x)\}.$$
We now recall the notion of \emph{generalized monotonicity}, which extends classical monotonicity by allowing the modulus to be negative. This weaker condition enables the study of a broader class of operators. 
\begin{definition} \cite{Minh} 
	An operator   $F: H \longrightarrow 2^{H}$ is said to be  \emph{ $\mu_F$ - monotone} if there exists  $\mu_F \in \R$ such that $\langle x-y, u-v\rangle \geq \mu_F\|x-y\|^2$ for all $x, y\in H, u\in F(x), v\in F(y)$. 
\end{definition}
\begin{remark} Note that in the definition above, unlike the classical definition, here we do not require that $\mu_F\geq 0$. In fact, if $\mu_F<0$,  $F$ is said to be \emph{weakly-monotone}. When $\mu_F=0$, $\mu_F$-monotonicity reduces to the classical monotonicity. If $\mu_F>0$, an $\mu_F$-monotone operator $F$  becomes  strongly monotone.   
\end{remark}

\begin{definition}\cite{Minh} 
	A  $\mu_F$-monotone operator $F$ is said to be  \emph{maximal} if there is no $\mu_F$-monotone operator whose graph strictly contains the graph of  $F$.  
\end{definition}

Here, we recall an important notion of Lipschitz continuity, which frequently appears in the study of algorithms.

\begin{definition}
	An operator $G~: H\longrightarrow H$ is said to be  \emph{Lipschitz continuous} with constant \(L\geq 0\) if $\|G(x)-G(y)\|\leq L\|x-y\|$ for all \(x,y\in H\).
\end{definition}
The resolvent of an operator is a fundamental tool in the study of inclusion problems. We now recall its definition. The resolvent of an operator $F$ with the parameter $\gamma$  is given by  
$$J_{\lambda F} =(Id+\gamma F)^{-1},$$
where $Id$ is  the \emph{identity mapping}.

In the absence of monotonicity, the resolvent may not always return a unique value at a given point. However, the following lemma establishes that for generalized monotone operators, the resolvent remains single-valued under suitable parameter choices. Furthermore, it demonstrates that the resolvent is cocoercive, a property that will play a crucial role in the subsequent analysis.

\begin{lemma}\label{GENMON} \cite{Minh, BauschkeCombettes} Let $F : H \longrightarrow 2^{H}$ be an $\mu_F$-monotone operator and let $\gamma  > 0$ be such that $1+\gamma \mu_F>0$. Then, the following hold.
	\begin{enumerate}
		\item $J_{\gamma F}$ is a singleton;
		\item $\mbox{\rm ran } J_{\gamma F} =\mbox{\rm dom } F$; 
        \item \label{tc4}$F$ is (maximal) \(\mu_F\)-monotone if and only if $F'=F-\mu_F id$ is (maximal) monotone. 
	\end{enumerate}
\end{lemma}

It is well known that the sum of two maximal monotone operators is not always maximal. The following lemma provides a criterion for determining when the sum of two operators remains maximal. We begin by recalling a classical result and then extend it to the setting of generalized monotonicity.

\begin{lemma}\cite{LEM}\label{sum-monotone}
	Let $F: H\longrightarrow 2^H$ be maximal monotone and $G: H \longrightarrow H$ be monotone and Lipschitz continuous on $H$. Then $F+G$ is maximally monotone.  
\end{lemma}
\begin{lemma}\label{lm2.2} 
	Let $F: H\longrightarrow 2^H$ be maximally $\mu_F$-monotone and $G: H \longrightarrow H$ be $\mu_G$-monotone and Lipschitz continuous on $H$. Then $F+G$ is maximally $ (\mu_F+\mu_G)$-monotone.  
\end{lemma}
\begin{proof}
	Let $\gamma>0$ such that $1+\gamma.\mu_F>0$. Because $F$ is maximally -$\mu_F$ monotone, it holds that $F':=F-\mu_F id$ is maximally monotone \cite{Minh}. Also, $G$ is $\mu_G$ monotone, Lipschitz continuous, it follows that $G'=G-\mu_G id $ is monotone, and Lipschit continuous. By Lemma \ref{sum-monotone}, $F'+G'=F+G-( \mu_F+\mu_G) id$ is maximally monotone. It follows from Part \eqref{tc4} of Lemma \ref{GENMON} that  $F+G$ is maximal $ (\mu_F+\mu_G)$-monotone. 
\end{proof}

The next lemma provides a characterization of when a generalized monotone operator is maximal. This extends the classical result. 

\begin{lemma}\label{max sum}
	Let $F: H\longrightarrow 2^H$ be a $\mu_F$-monotone operator. Then $F$ is the maximal monotone if and only if 
	\begin{equation}\label{dk max mon}
		\forall (v, y)\in \mbox{\rm gr}(F), \langle u-v, x-y\rangle \geq \mu_F\|x-y\|^2 \implies u\in F(x).
	\end{equation}  
\end{lemma}
\begin{proof}
	Suppose that $F$ is maximal $\mu_F$-monotone and $u_0, x_0\in H$ such that for all $(v, y)\in \mbox{\rm gr} F, \langle u_0-v, x_0-y\rangle \geq0$. We now suppose contradiction that $u_0\not\in F(x_0)$. Let 
	\begin{equation*}
		T(x)=\begin{cases}
			F(x) &\mbox{ if } x\neq x_0\\
			F(x)\cup \{u_0\} & \mbox{ otherwise} 
		\end{cases}
	\end{equation*}
	Then $T$ is $\mu_F$-monotone and $\gr \ F \subset gr\ T$, a contradiction to the maximality of $F$. Hence, $u_0\in F(x_0)$.
	
	Assume now that $u, x\in H$ satisfies condition \eqref{dk max mon}. Let  $A:H \longrightarrow 2^H $ be a $\mu_F$-monotone such that $\gr F\subseteq \gr A.$ Then for all $(u,x)\in \gr A$, by the $\mu_F$-monotone of $A$ we have that $\langle u-v, x-y \rangle \geq \mu_F\|x-y\|^2$ for all $(v, y)\in \gr A$. Since $\gr \ F\subseteq \gr \ A$, this also holds for all $(v,y)\in \gr F$. By condition \eqref{dk max mon} we derive that $u\in F(x)$ or $(u,x)\in \gr F$. This implies that $F$ is maximal $\mu_F$-monotone.  
\end{proof}

For a comprehensive discussion on monotone operators, their applications in optimization problems, and the properties of their resolvent, we refer readers to \cite{AVR, BauschkeCombettes, Minh}.

\section {Algorithm and Convergence analysis}\label{sec3}
In this section we first present an algorithm proposed in \cite{CHIN}, we then analyze the property of strong convergence for the sequence generated by the algorithm. We emphasize that in assumption \ref{gt1} operators are assumed to be generalized monotone, which is weaker than classical monotonicity. 
\begin{algorithm}\label{alg1}
Let $\gamma \in \left(0, \frac{1}{2L}\right), \theta_1\in [0,1), \ \theta_2\leq 0$ and take $\{\lambda_k\}\subseteq (0, 1)$. For any $w^*, u_{-1}, u_0, u_1\in H$, and suppose that $u_k, u_{k-1}, u_{k-2}$ are given. Set 
\begin{align}\label{eqalg}
    u_{k+1}=&J_{\gamma F} (\lambda_k w^*+(1-\lambda_k)\big(u_k-\theta_1(u_k-u_{k-1})+\theta_2(u_{k-1}-u_{k-2}) \notag \\
    & -\gamma Gu_k -\gamma (1-\lambda_k)(Gu_k-Gu_{k-1})\big), \quad \forall k\geq 1. 
\end{align}
\end{algorithm} 
To achieve strong convergence in this algorithm, we impose the following assumptions on the operators and parameters:
\begin{assumption}\label{gt1}
	\begin{enumerate}
		\item \label{dk1 gt1}\(F\) is maximal $-\mu_F$ monotone;
		\item \label{dk2 gt1} $G$ is $\mu_G$-monotone and Lipschitz continuous with constant $L>0$;
		\item \label{dk5 gt1} $ \mu_F+\mu_G\geq 0$;
		\item \label{dk3 gt1} $\zer(F+G)\neq \emptyset $; 
		\item \label{dk4 gt1} $\theta_1, \theta_2$ satisfy $0\leq \theta_1 <\frac{1}{3}(1-2\gamma L),\quad  \frac{1}{3+4\theta_1} \left(3\theta_1-1+2\gamma L<\theta_2 \leq 0\right)$;
        \item $1+\gamma \mu_F>0$.
	\end{enumerate}
\end{assumption}
\begin{remark}
It is worth noting that condition \eqref{dk5 gt1} in this assumption is weaker than the monotonicity assumption. Indeed, it may happen that $F$ or $G$ is weakly monotone, while the sum of the two operators is monotone or strongly monotone. 
\end{remark}

\begin{lemma}\label{lm3.4}
	Assume that Assumption \ref{gt1} holds. Then the sequence \(\{u_k\}\) generated by the Algorithm \ref{alg1} is bounded whenever $\lim\limits_{k\to \infty} \lambda_k=0$.
\end{lemma}
\begin{proof}
	Let \(u^*\in zer(F+G)\)  and set $z_k=\lambda w^*+(1-\lambda_k)w_k$ with $w_k=u_k+\theta_1(u_k-u_{k-1})+\theta_2(u_{k-1}-u_{k-2})$. Then 
	\begin{equation}\label{eq3.2}
		-\gamma Gu^*\in \gamma F u^*
	\end{equation}
	and
	\begin{equation}\label{eq3.3}
		z_k-\gamma G u_k-\gamma(1-\lambda_k)(G u_k-Gu_{k-1})-u_{k+1}\in \gamma  F u_{k+1}.        
	\end{equation}
	By the $\mu_F$-monotonicity of $F$ it follows from \eqref{eq3.2} and \eqref{eq3.3} that
	\begin{equation*}
		\langle z_k -\gamma G u_k -\gamma (1-\lambda_k)(G u_k-Gu_{k-1})-u_{k+1}+\gamma G u^*, u_{k+1}-u^*\geq 2\gamma \mu_F \|u_{k+1}-u^*\|^2.
	\end{equation*}  
	Consequently, 
	\begin{align}
		2\gamma \mu_F \|u_{k+1}-u^*\|^2 & \leq 2\langle u_{k+1}-z_k+\gamma G u_k +\gamma (1-\lambda_k)(G u_k-Gu_{k-1})-\gamma G u^*, u^*-u_{k+1}\rangle \notag\\
		& = 2 \langle u_{k+1}-z_k, u^*-u_{k+1}\rangle +2\gamma \langle Gu_k-Gu^*, u^*-u_{k+1}\rangle +2\gamma (1-\lambda_k) \notag\\
		&\langle G u_k-Gu_{k-1}, u^*-u_k\rangle +2\gamma (1-\lambda_k)\langle Gu_k-G_{k-1}, u_k-u_{k+1}\rangle \notag      
	\end{align}
	By equation \eqref{eq2.1} it follows that 
	\begin{align}\label{eq3.4}
		2\gamma \mu_F \|u_{k+1}-u^*\|^2 & \leq  \|z_k-u^*\|^2-\|u_{k+1}-u^*\|-\|u_{k+1}-z_k\|^2+2\gamma \langle Gu_k-Gu^*, u^*-u_{k+1}\rangle \notag \\
		&+2\gamma (1-\lambda_k)\langle Gu_k-Gu_{k-1}, u^*-u_k\rangle +2\gamma (1-\lambda_k)\langle G u_k-Gu_{k-1}, u_k-u_{k+1}\rangle 
	\end{align}
	Because of $\mu_G$-monotonicity of $G$ one has 
	\begin{align}\label{eq3.5}
		& \langle G u_k-G u^*, u^*-u_{k+1} \rangle = \langle G u_k-Gu_{k+1}+Gu_{k+1}-G u^*, u^*-u_{k+1} \rangle \notag \\
		=&   \langle G u_k-Gu_{k+1},  u^*-u_{k+1} \rangle + \langle Gu_{k+1}-G u^*, u^*-u_{k+1} \rangle \notag \\
		\leq & \langle G u_k-Gu_{k+1},  u^*-u_{k+1} \rangle  -\mu_G\|u_{k+1}-u^*\|^2.
	\end{align}
	In addition, $G$ is Lipschit continuous with modulus $L$, we have
	\begin{align}\label{eq3.6}
		2\gamma \langle G u_k-Gu_{k-1}, u_k-u_{k+1}\rangle \leq & 2 \gamma \|Gu_k-Gu_{k-1}\|u_k-u_{k+1}\| \notag \\
		\leq & 2\gamma L \|u_k-u_{k-1}\| \|u_k-u_{k+1}\| \notag \\
		\leq & \gamma L \left(\|u_k-u_{k-1}\|^2+\|u_k-u_{k+1}\|^2\right).
	\end{align}
	Substituting \eqref{eq3.5} and \eqref{eq3.6} into \eqref{eq3.4} we get 
	\begin{align} \label{eq3.7} 
		& (2\gamma \mu_F+2\gamma \mu_G)\|u_{k+1}-u^*\|^2 +\|u_{k+1}-u^*\|^2+2\gamma \langle Gu_{k+1}-Gu_k, u^*-u_{k+1}\rangle \notag \\
		\leq & \|z_k-u^*\|^2-\|u_{k+1}-z_k\|^2+2\gamma (1-\lambda_k)\langle Gu_k-Gu_{k-1}, u^*-u_k\rangle  \notag \\
		&+(1-\lambda_k) \gamma L \left(\|u_k-u_{k-1}\|^2+\|u_k-u_{k+1}\|^2\right), \forall k\geq k_0.
	\end{align}
	By Lemma \ref{lm2.1} one has 
	\begin{align}\label{eq3.8}
		\|z_k-u^*\|^2=& \|w_k-u^*-\lambda_k(w_k-w^*)\|^2=\|w_k-u^*\|^2+\lambda_k^2\|w_k-w^*\|^2-2\lambda_k\langle w_k-u^*, w_k-w^*\rangle \notag \\
		= & \|w_k-u^*\|^2+\lambda_k^2\|w_k-w^*\|^2-\lambda_k\|w_k-w^*\|^2-\lambda_k\|w_k-u^*\|^2+\lambda_k\|w^*-u^*\|^2.
	\end{align}
	Substituting $u^*$ by $u_{k+1}$ in \eqref{eq3.8} one obtains
	\begin{align}\label{eq3.9} 
		\|z_k-u_{k+1}\|^2=& \|w_k-u_{k+1}\|^2+\lambda_k^2\|w_k-w^*\|^2-\lambda_k\|w_k-w^*\|^2-\lambda_k\|w_k-u_{k+1}\|^2\notag \\
		& +\lambda_k\|w^*-u_{k+1}\|^2.
	\end{align}
	Difference of \eqref{eq3.8} and \eqref{eq3.9} yields 
	\begin{align}\label{eq3.10}
		&\|z_k-u^*\|^2-\|w_k-u_{k+1}\|^2 \notag \\
		=& (1-\lambda_k)\|w_k-u^*\|^2+\lambda_k\|w^*-u^*\|^2-(1-\lambda_k)\|u_{k+1}-w_k\|^2-\lambda_k\|u_{k+1}-w^*\|.
	\end{align}
	Combining \eqref{eq3.10} and \eqref{eq3.7} we get 
	\begin{align}\label{eq3.11}
		& (2\gamma \mu_F+2\gamma\mu_G)\|u_{k+1}-u^*\|^2+\|u_{k+1}-u^*\|^2+2\gamma \langle Gu_{k+1}-Gu_k, u^*-u_{k+1}\rangle \notag \\
		\leq & (1-\lambda_k)\|w_k-u^*\|^2+\lambda_k\|w^*-u^*\|^2-(1-\lambda_k)\|u_{k+1}-w_k\|^2-\lambda_k\|u_{k+1}-w^*\|^2 \notag \\
		& +2\gamma (1-\lambda)\langle Gu_k-Gu_{k-1}, u^*-u_k\rangle + (1-\lambda_k)\gamma L   \left(\|u_k-u_{k-1}\|^2+\|u_k-u_{k+1}\|^2\right)
	\end{align}
    for all \( k\geq k_0.\)
	By \eqref{eq2.2} we derive that
	\begin{align}\label{eq3.12} 
		\|w_k-u^*\|^2=& \|u_k+\theta_1(u_k-u_{k-1})+\theta_2(u_k-1-u_{k-1})-u^*\|^2 \notag \\
		=& \|(1+\theta_1)(u_k-u^*)-(\theta_1-\theta_2)(u_{k-1}-u^*)-\theta_2(u_{k-2}-u^*\|^2\notag \\
		&+(1+\theta_1)(\theta_1-\theta_2)\|u_k-u_{k-1}\|^2+\theta_2(1+\theta_1)\|u_k-u_{k-2}\|^2.
	\end{align}
	In addition, by \eqref{eq2.1} it holds that
	\begin{align}\label{eq3.13} 
		\|u_{k+1}-w_k\|^2=& \|u_{k+1}-u_k\|^2-2\langle u_{k+1}-u_k, \theta_1(u_k-u_{k-1})+\theta_2(u_{k-1}-u_{k-2})\notag \\
		&+ \|\theta_1(u_k-u_{k-1})+\theta_2(u_{k-1}-u_{k-2}\|^2\notag \\
		=& \|u_{k+1}-u_k\|-2\theta_1\langle u_{k+1}-u_k, u_k-u_{k-1}\rangle \notag\\
		&-2\theta_2 \langle u_{k+1}-u_k, u_{k-1}-u_{k-2} \rangle +\theta_1^2\|u_k-u_{k-1}\|^2\notag\\
		&+2\theta_2\theta_1 \langle u_k-u_{k-1}, u_{k-1}-u_{k-2}\rangle +\theta_2^2\|u_{k-1}-u_{k-2}\|^2.
	\end{align}
	Moreover, note that 
	\begin{align}\label{eq3.14}
		-2\theta_1\langle u_{k+1}-u_k, u_k-u_{k-1}\rangle \geq & -2\theta_1 \|u_{k+1}-u_k\|\|u_k-u_{k-1}\|\notag\\
		\geq & -\theta_1 \|u_{k+1}-u_k\|^2-\theta_1\|u_k-u_{k-1}\|^2,
	\end{align}
	\begin{align}\label{eq3.15}
		-2\theta_2 \langle u_{k+1}-u_k, u_{k-1}-u_{k-2}\rangle \geq & -2|\theta_2| \|u_{k+1}-u_k\|\|u_{k-1}-u_{k-2}\|\notag\\
		\geq & -|\theta_2| \|u_{k+1}-u_k\|^2-|\theta_2|\|u_{k-1}-u_{k-2}\|^2,
	\end{align}
	\begin{align}\label{eq3.16}
		2\theta_2\theta_1 \langle u_k-u_{k-1},u_{k-1}-u_{k-2}\rangle \geq & -2|\theta_2| |\theta_1| \|u_k-u_{k-1}\|u_{k-1}-u_{k-2}\|\notag\\
		\geq &-|\theta_2||\theta_1| \|u_k-u_{k-1}\|^2-|\theta_2||\theta_1| \|u_{k-1}-u_{k-2}\|^2.
	\end{align}
	Plugging \eqref{eq3.14}, \eqref{eq3.15}, and \eqref{eq3.16} into \eqref{eq3.13} yields
	\begin{align}\label{eq3.17}
		\|u_{k+1}-w_k\|^2\geq &\|u_{k+1}-u_k\|^2-\theta_1\|u_{k+1}-u_k\|^2 -\theta_1\|u_{k}-u_{k-1}\|^2\notag \\
		& -|\theta_2|\|u_{k+1}-u_k\|^2 -|\theta_2|\|u_{k_1}-u_{k-2}\|^2 +\theta_1^2 \|u_k-u_{k-1}\|^2\notag\\
		=& (1-\theta_1-|\theta_2|)\|u_{k+1}-u_k\|^2+(\theta_1^2-\theta_1-\theta_1|\theta_2|)\|u_{k}-u_{k-1}\|^2\notag\\
		& + (\theta_2^2-|\theta_2|-\theta_1|\theta_2|)\|u_{k-1}-u_{k-2}\|^2.
	\end{align}
	Again, plugging \eqref{eq3.12} and \eqref{eq3.17} into \eqref{eq3.11} we get 
	\begin{align*}
		& (2\gamma \mu_F+2\gamma\mu_G)\|u_{k+1}-u^*\|^2+\|u_{k+1}-u^*\|^2+2\gamma \langle Gu_{k+1}-Gu_k, u^*-u_{k+1}\rangle \notag\\
		\leq & (1-\lambda_k)\left[ (1+\theta_1 )\|u_k-u^*\|^2-(\theta_1-\theta_2)\|u_{k-1}-u^*\|^2-\theta_2\|u_{k-2}-u^*\|^2\right. \notag \\
		& \left. (1+\theta_1  )(\theta_1-\theta_2)\|u_k-u_{k-1}\|^2+\theta_2(1+\theta_1)\|u_k-u_{k-2}\|^2-\theta_2(\theta_1-\theta_2)\|u_{k-1}-u_{k-2}\|^2\right] \notag\\
		&+\lambda_k\|w^*-u^*\|^2-(1-\lambda_k)\left[ (1-\theta_1-|\theta_2|\|u_{k-1}-u_k\|^2+(\theta_1^2-\theta_1-\theta_1|\theta_2|\|u_k-u_{k-1}\|^2\right.\notag\\
		&\left. + (\theta_2^2-|\theta_2|-\theta_1|\theta_2|)\|u_{k-1}-u_{k-2}\|^2\right]-\lambda_k\|u_{k+1}-w^*\|^2\notag \\
		& +2\gamma (1-\lambda_k)\langle G u_k-Gu_{k-1}, u^*-u_k \rangle +(1-\lambda_k)\gamma L \left(\|u_k-u_{k-1}\|^2+\|u_{k+1}-u_k\|^2\right)\notag\\
		\leq & (1-\lambda_k)\left[ (1+\theta_1 )\|u_k-u^*\|^2-(\theta_1-\theta_2)\|u_{k-1}-u^*\|^2-\theta_2\|u_{k-2}-u^*\|^2\right.\notag \\
		&+(2\theta_1-\theta_2-\theta_1\theta_2 +\theta_1|\theta_2|)\|u_k-u_{k-1}\|^2+(|\theta_2|+|\theta_2|\theta_1-\theta_2\theta_1)\|u_{k-1}-u_{k-2}\|^2\notag\\
		& \left. -(1-\theta_1-|\theta_2|)\|u_{k+1}-u_k\|^2+2\gamma \langle G u_k-Gu_{k-1}, u^*-u_k\rangle \right]\notag\\
		&+\lambda_k\|w^*-u^*\|^2+(1-\lambda_k)\gamma L \left(\|u_k-u_{k-1}\|^2+\|u_{k+1}-u_{k}\|^2\right), \forall k\geq k_0.
	\end{align*}
	It follows that
	\begin{align}\label{eq3.18}
		&(1+2\gamma \mu_F+2\gamma\mu_G)\|u_{k+1}-u^*\|^2-\theta_1\|u_k-u^*\|^2-\theta_2\|u_{k-1}-u^*\|^2+2\gamma \langle G u_{k+1} -Gu_k, u^* -u_{k+1} \rangle \notag  \\
		&+(1-|\theta_2| -\theta_1 -\gamma L )\|u_{k+1}-u_{k}\|^2 \leq (1-\lambda_k) \left[\|u_k-u^*\|^2-\theta_1\|u_{k-1}-u^*\|^2\right.\notag \\
		& -\theta_2\|u_{k-2}-u^*\|^2+2\gamma \langle G u_k-G u_{k-1}, u^*-u_k\rangle \notag \\
		&\left. +(1-|\theta_2|-\theta_1 -\gamma L )\|u_k-u_{k-1}\|^2\right] +\lambda_k\|w^*-u^*\|^2 \notag \\
		&+(1-\lambda_k)\left[ 2\gamma L+3\theta_1-1+(1+\theta_1)(|\theta_2|-\theta_2)\right] \|u_k-u_{k-1}\|^2\notag \\
		&+(1-\lambda_k)(|\theta_2|+|\theta_2|\theta_1-\theta_2\theta_1)\|u_{k-1}-x_{k-2}\|^2\notag\\
		= & (1-\lambda_k)\left[ \|u_k-u^*\|^2-\theta_1\|u_{k-1}-u^*\|^2-\theta_2 \|u_{k-2}-u^*\|^2 +2\gamma \langle Gu_k -Gu_{k-1}, u^*-u_k\rangle \right. \notag\\
		&+(1-|\theta_2|-\theta_1 -\gamma L)\|u_k-u_{k-1}\|^2 \notag \\
		& \left. -(2\gamma L +3\theta_1 -1 +(1+\theta_1)(|\theta_2|-\theta_2)\left(\|u_{k-1}-u_{k-2}\|^2 -\|u_k-u_{k-1}\|^2\right)\right] +\lambda_k \|w^*-u^*\|^2\notag\\
		&-(1-\lambda_k)\left[ -(2\gamma L +3\theta_1 -1+(1+\theta_1)(|\theta_2|-\theta_2)-(|\theta_2|+|\theta_2|\theta_1 -\theta_2\theta_1)\right]\|u_{k-1}-u_{k-2}\|^2.
	\end{align}
	Set $a_1=-(2\gamma L+3\theta_1-1+(1+\theta_1)(|\theta_2|-\theta_2)),\ a_2=1-3\theta_1 -2\gamma L-2|\theta_2|-2\theta_1 |\theta_2| +\theta_2 +2\theta_1\theta_2$ and 
	\begin{align*}
		q_k= & (1+2\gamma \mu_F+2\gamma\mu_G)\|u_k-u^*\|^2-\theta_1 \|u_{k-1}-u^*\|^2 -\theta_2 \|u_{k-2}-u^*\|+2\gamma \langle Gu_k-Gu_{k-1}, u^*-u_k\rangle \notag \\
		&+(1-|\theta_2|-\theta_1-\gamma L)\|u_k-u_{k-1}\|^2+a_1\|u_{k-1}-u_{k-2}\|^2.
	\end{align*}
	Then \eqref{eq3.18} reads as follows.
	\begin{align}\label{eq3.19}
		q_{k+1}\leq & (1-\lambda_k)q_k+\lambda_k\|u^*-w^*\|^2-(1-\lambda_k)a_2\|u_{k-1}-u_{k-2}\|^2 -(1-\lambda_k)(2\gamma \mu_F+2\gamma\mu_G)\|u_k-u^*\|^2 \notag \\
		\leq & (1-\lambda_k)q_k+\lambda_k\|u^*-w^*\|^2-(1-\lambda_k)a_2\|u_{k-1}-u_{k-2}\|^2
	\end{align} 
	for all $ k\geq k_0.$ The last inequality holds since $\lambda_k\in (0,1)$ and condition \eqref{dk5 gt1} in Assumption \ref{gt1}. 
	
	We now prove that $a_1>0, a_2>0$ and $q_k\geq 0$ for all $k=1,2,\dots $. By condition \eqref{dk4 gt1} in Assumption \ref{gt1} we have $3\theta_1-1+2\gamma L<0$. Therefore, $\dfrac{1}{2+2\theta_1}(3\theta_1-1+2\gamma L <\dfrac{1}{3+4\theta_1} (3\theta_1-1+2\gamma L<\theta_2.$ This implies that $$3\theta_1-1+2\gamma L -2\theta_2 -2\theta_1\theta_2<0.$$
Because $|\theta_2|=-\theta_2$, it holds that 
	\begin{equation}\label{eq3.20} 
		3\theta_1 -1+2\gamma L +|\theta_2|-\theta_2 +\theta_1|\theta_2|-\theta_1\theta_2<0.
	\end{equation}
	As a result, $a_1>0$.
	
	Next, by the assumption $\dfrac{1}{3+4\theta_1}(3\theta_1-1+2\gamma L)<\theta_2$ we get 
    $$1-3\theta_1 -2\gamma L +3\theta_2 +4\theta_2\theta_1>0.$$	
	Again, because $|\theta_2|=-\theta_2$ it holds that 
	\begin{equation}\label{eq3.21}
		1-3\theta_1 -2\gamma L -2|\theta_2| +\theta_2 -2|\theta_2|\theta_1 +2\theta_2\theta_1>0
	\end{equation}
	which implies that $a_2>0$. 
	
	To show that $q_k\geq 0$ for all $k$, note that $\theta_2\leq 0,$ and $a_1>0$. Hence for all $k\geq k_0$ we have 
	\begin{align}\label{eq3.22} 
		q_k\geq & \|u_k-u^*\|^2-\theta_1 \|u_{k-1}-u^*\|^2+2\gamma \langle Gu_k-Gu_{k-1}, u^*-u_k\rangle +(1-|\theta_2|-\theta_1-\gamma L)\|u_k-u_{k-1}\|^2\notag \\
		\geq & \|u_k-u^*\|^2-\theta_1\|u_{k-1}-u^*\|^2-\gamma L (\|u_k-u_{k-1}\|^2+\|u_k-u^*\|^2)\notag \\
		& +(1-|\theta_2|-\theta_1 -\gamma L)\|u_k-u_{k-1}\|^2 \notag  \\     
		\geq & (1-\gamma L) \|u_k-u^*\|^2-\theta_1 (2\|u_k-u_{k-1}\|^2+2\|u_k-u^*\|^2) +(1-|\theta_2|-\theta_1 -2\gamma L)\|u_k-u_{k-1}\|^2\notag \\
		=& (1-2\theta_1 -\gamma L)\|u_k-u^*\|^2+(1-|\theta_2| -3\theta_1 -2\gamma L)\|u_k-u_{k-1}\|^2\notag \\
		\geq&  (1-3\theta_1 -\gamma L)\|u_k-u^*\|+(1-|\theta_2|-\theta_1 -2\gamma L)\|u_k-u_{k-1}\|^2
	\end{align}
	Because $\theta_1 <\dfrac{1}{3}(1-2\gamma L)$ one has $3\theta_1 -1+2\gamma L<0$. It follows that $1-3\theta_1 -\gamma L>0$ and $3\theta_1 -1+2\gamma <\dfrac{1}{3+4\theta_1}(3\theta_1-1+2\gamma L)<\theta_2$. Thus, $-|\theta_2|-3\theta_1+1-2\gamma L>0$. From \eqref{eq3.22} we derive $q_k\geq 0$ for all $k\geq k_0$. 
	
	In addition, because $a_2>0$ and $\lambda_k\in (0,1)$, from \eqref{eq3.19} we have 
	$$q_{k+1}  \leq (1-\lambda_k) q_k +\lambda_k\|u^*-w^*\|^2.$$ So Lemma \ref{lm2.3} can be invoked with $b_k=0$ for all $k$ and $a_k=\lambda\|u^*-w^*\|^2$. Hence $\{q_k\}$ is bounded. Therefore, from \eqref{eq3.22} we derive that the sequence $\{u_k\}$ is also bounded as asserted. 
\end{proof}
\begin{theorem}
	Suppose that Assumption \ref{gt1} holds. Suppose further that $\lim\limits_{k\to \infty} \lambda_k=0$ and $\sum\limits_{k=1}^\infty \lambda_k=\infty$. Then the sequence $\{u_k\}$ generated by Algorithm \ref{alg1} converges strongly to $P_{\mbox{\rm zer}{(F+G)}} w^*.$
\end{theorem}
\begin{proof}
	Let $u^*=_{{\mbox{\rm} zer}{(F+G)}} w^*$. Then using \eqref{eq2.1} we have 
	\begin{align}\label{eq3.23}
		\|z_k-u^*\|=& \|\lambda_k(w^*-u^*)+(1-\lambda_k)(w_k-u^*)\|^2\notag \\
		=& \lambda^2_k\|w^*-u^*\|^2+(1-\lambda_k)^2\|w_k-u^*\|^2+2\lambda_k(1-\lambda_k)\langle w^*-u^*, w_k-u^*\rangle.
	\end{align}
	Once again, by \eqref{eq2.1} we have   
	\begin{align}\label{eq3.24}
		\|z_k-u_{k+1}\|^2=& \lambda_k^2\|w^*-u_{k+1}\|^2+(1-\lambda_k)^2\|w_k-u_{k+1}\|^2+2\lambda_K(1-\lambda_k)\langle w^*-u_{k+1}, w_k-u_{k+1}\rangle \notag \\
		\geq & \lambda_k^2\|x_{k+1}-w^*\|^2+(1-\lambda_k)^2\|u_{k+1}-w_k\|^2-2\lambda_k(1-\lambda_k)\|u_{k+1}-w^*\|\|u_{k+1}-w_k\|\notag \\ 
		\geq & \lambda_k^2 \|u_{k+1}-w^*\|^2+(1-\lambda_K)^2\|u_{k+1}-w_k\|^2-2\lambda_k(1-\lambda_k)\Gamma \|u_{k+1}-w_k\|,
	\end{align}
	where $\Gamma=\sup\limits_{k\geq 1} \|u_{k+1}-w^*\|$. Note that this supremum exists because the sequence $\{u_k\}$ is bounded due to Lemma \ref{lm3.4}. 
	
	Putting \eqref{eq3.23} and \eqref{eq3.24} in \eqref{eq3.7} we have 
	\begin{align}\label{eq3.25} 
		&(1+2\gamma \mu_F+2\gamma\mu_G)\|u_{k+1}-u^*\|^2+2\gamma \langle G u_{k+1}-Gu_k, u^*-u_{k+1}\rangle \notag \\
		\leq & \lambda^2_k\|u^*-w^*\|+(1-\lambda_k)^2\|w_k-u^*\|^2+2\lambda_k(1-\lambda_k) \langle w^*-u^*, w_k-u^*\rangle \notag \\
		& - (\lambda^2_k\|u_{k+1}-w^*\|+(1-\lambda_k)^2\|u_{k+1}-w_k\|^2-2\lambda_k(1-\lambda_k)\Gamma \|u_{k+1}-w_k\|)\notag \\
		& +2\gamma (1-\lambda_k)\langle Gu_k-Gu_{k-1}, u^*-u_k\rangle\notag \\
		& +(1-\lambda_k)\gamma L(\|u_k-u_{k-1}\|^2+\|u_{k+1}-u_k\|^2)\notag \\ 
		\leq & (1-\lambda_k)\left(\|w_k-u^*\|^2+2\gamma \langle Gu_k-Gu_{k-1}, u^*-u_k\rangle \right) \notag \\
		& +\lambda_k (\lambda_k\|u^*-w^*\|^2+2(1-\lambda_k)\langle w^*-u^*, w_k-u^*\rangle +2(1-\lambda_k)\Gamma \|u_{k+1}-w_k\|)\notag \\
		&-(1-\lambda_k)^2\|u_{k+1}-w_k\|^2+(1-\lambda_k)\gamma L (\|u_k-u_{k-1}\|^2+\|u_{k+1}-u_k\|^2), \forall  k\geq k_0.
	\end{align} 
	Now plugging \eqref{eq3.12} and \eqref{eq3.17} in \eqref{eq3.25} we get 
	\begin{align*}
		&(1+2\gamma \mu_F+2\gamma\mu_G)\|u_{k+1}-u^*\|^2+2\gamma \langle Gu_{k+1}-Gu_k, u^*-u_{k+1}\rangle \notag \\
		\leq & (1-\lambda_k)\left[(1+\theta_1)\|u_k-u^*\|^2-(\theta_1-\theta_2)\|u_{k-1}-u^*\|^2 -\theta_2\|u_{k-2}-u^*\|^2 \right.\notag \\
		&+(1+\theta_1)(\theta_1-\theta_2)\|u_k-u_{k-1}\|^2 +\theta_2 (1+\theta_1 )\|u_k-u_{k-2}\|^2\notag\\
		&\left. -\theta_2(\theta_1-\theta_2) \|u_{k-1}-u_{k-2}\|^2+2\gamma \langle Gu_k-Gu_{k-1}, u^*-u_k\rangle \right]\notag \\
		&+\lambda_k(\lambda_k(u^*-w^*\|^2+2(1-\lambda_k)\langle w^*-u^*, w_k-u^*\rangle +2(1-\lambda_k)\Gamma \|u_{k+1}-w_k\|)\notag \\
		& -(1-\lambda_k)^2\left[(1-\theta_1-|\theta_2|)\|u_{k+1}-u_k\|^2+(\theta_1^2-\theta_1 -\theta_1|\theta_2|)\|u_k-u_{k-1}\|^2\right.\notag \\
		&\left. +(\theta_2^2-|\theta_2|-\theta_1|\theta_2|)\|u_{k-1}-x_{k-2}\|^2\right]+ (1-\lambda_k)\gamma L (\|u_k-u_{k-1}\|^2+\|u_{k+1}-u_k\|^2).
	\end{align*}
	It follows that 
	\begin{align}\label{nam1} 
		&(1+2\gamma \mu_F+2\gamma\mu_G)\|u_{k+1}-u^*\|^2-\theta_1\|u_k-u^*\|^2-\theta_2\|u_{k_1}-u^*\|^2+2\gamma \langle Gu_{k+1}-Gu_k, u^*-u_{k+1}\rangle \notag \\
		&+(1-|\theta_2|-\theta_1-\gamma L)\|u_{k+1}-u_k\|^2\notag \\
		\leq & (1-\lambda_k)\Big[ \|u_k-u^*\|^2-\theta_1\|u_{k-1}-u^*\|^2-\theta_2\|u_{k-2}-u^*\|^2+2\gamma \langle Gu_k-Gu_{k-1}, u^*-u_k\rangle\Big.\notag \\
		&\Big. +(1-|\theta_2|-\theta_1 -\gamma L)\|u_k-u_{k-1}\|^2\Big] \notag \\
		&+\lambda_k\Big(\lambda_k\|u^*-w^*\|+2(1-\lambda_k)\langle w^*-u^*, w_k-u^*\rangle +2(1-\lambda_k)\Gamma \|u_{k+1}-w_k\|\Big)\notag \\
		&+ (1-\lambda_k) \Big[2\gamma L+2\theta_1-\theta_2 -\theta_1\theta_2 -1+|\theta_2|+\theta_1^2-(1-\lambda_k)(\theta_1^2-\theta_1 -\theta_1|\theta_2|) \Big].
	\end{align}
	Also,
	\begin{align}\label{nam2} 
		&\|u_k-u_{k-1}\|^2 +(1-\lambda_k)\Big[\theta_2^2-\theta_2\theta_1 -(1-\lambda_k)(\theta_2^2-|\theta_2|-\theta_1|\theta_2|\Big] \|u_{k-1}-u_{k-2}\|^2 \notag \\
		=& (1-\lambda_k)\Bigg[ \|u_k-u^*\|^2-\theta_1 \|u_{k-1}-u^*\|^2-\theta_2 \|u_{k-2}-u^*\|^2+2\gamma \langle Gu_k-Gu_{k-1}, u^*-u_k\rangle \Big. \notag \\
		&+ (1-|\theta_2|-\theta_1-\gamma L)\|u_k-u_{k-1}\|^2\notag  \\
		& -\Big(2\gamma L+2\theta_1 -\theta_2-\theta_1\theta_2 -1+|\theta_2|+\theta_1^2-(1-\lambda_k)(\theta_1^2-\theta_1-\theta_1|\theta_2|)\Big) \notag \\
		&\Bigg.\Big(\|u_{k-1}-u_{k-2}\|^2-\|u_k-u_{k-1}\|^2\Big)\Bigg] \notag \\
		&+\lambda_k\Big(\lambda_k(\|u^*-w^*\|+2(1-\lambda_k)(w^*-u^*, w_k-u^*\rangle +2(1-\lambda_k)\Gamma \|u_{k+1}-w_k\|\Big) \notag \\
		& -(1-\lambda_k)\Big[1-2\theta_1-2\gamma L+\theta_2+2\theta_1\theta_2 -|\theta_2|-\theta_2^2-\theta_1^2+(1-\lambda_k)](\theta_1^2-\theta_1 -\theta_1|\theta_2|) \notag \\
		&\Big. +(1-\lambda_k)(\theta_2^2-|\theta_2|-\theta_1|\theta_2|)\Big]\|u_{k-1}-u_{k-2}\|^2, \quad \forall k\geq k_0.
	\end{align}
	Set \begin{align*}
		r_k=& -\big(2\gamma L +2\theta_1 -\theta_2 -\theta_1\theta_2 -1+|\theta_2|+\theta_1^2-(1-\lambda_k)(\theta_1^2-\theta_1 -\theta_1|\theta_2|)\big),\\
		a_k=& \lambda_k\|w^*-u^*\|^2+2(1-\lambda_k)\langle w^*-u^*, w_k-u^*\rangle +2(1-\lambda_k)\Gamma \|u_{k+1}-w_k\|,\\
		s_{k+1} = & (1+2\gamma \mu_F+2\gamma\mu_G)\|u_k-u^*\|^2-\theta_1\|u_{k-1}-u^*\|^2-\theta_2\|u_{k-2}-u^*\|^2\notag \\
        &+2\gamma \langle Gu_k-Gu_{k-1}, u^*-u_k\rangle +(1-|\theta_2|-\theta_1 -\gamma L)\|u_k-u_{k-1}\|^2+r_k\|u_{k-1}-u_{k-2}\|^2,\\
		q_k=& 1-2\theta_1 -2\gamma L +\theta_2 +2\theta_1\theta_2 -|\theta_2|-\theta_2^2-\theta_1^2+(1-\lambda_k)(\theta_1^2-\theta_1-\theta_1|\theta_2|)\\
		&+(1-\lambda_k)(\theta_2^2-|\theta_2|-\theta_1|\theta_2|).
	\end{align*}
	Then from \eqref{nam1} and \eqref{nam2} we derive 
	\begin{align}\label{eq3.26} 
		s_{k+1}\leq &(1-\lambda_k)s_k+\lambda_k a_k -(1-\lambda_k)q_k\|u_{k-1}-u_{k-2}\|^2 -(1-\lambda_k)(2\gamma \mu_F+2\gamma\mu_G)\|u_{k+1}-u^*\|^2 \notag \\
		\leq &(1-\lambda_k)s_k+\lambda_k a_k -(1-\lambda_k)q_k\|u_{k-1}-u_{k-2}\|^2, \forall k\geq k_0
	\end{align}
	By \eqref{eq3.20} one has $1-3\theta_1 -2\gamma L -|\theta_2|+\theta_2-\theta_1|\theta_2|+\theta_1\theta_2 >0.$ This implies that 
	\begin{align*}
		\lim\limits_{k\to \infty} r_k=&\lim\limits_{k\to \infty} -\Big(2\gamma L++2\theta_1 -\theta_2 -\theta_1\theta_2 -1+|\theta_2|+\theta_1^2-(1-\lambda_k)(\theta_1^2-\theta_1 -\theta_1|\theta_2|)\Big)\\
		=& 1-3\theta_1 -2\gamma L -|\theta_2|+\theta_2 -\theta_1|\theta_2|+\theta_1\theta_2 >0.
	\end{align*}
	So, there is $k_1\geq k_0$ such that $r_{k}>0 $ for all $k\geq k_1$. Moreover, from \eqref{eq3.21} one has 
	\begin{equation*}
		1-3\theta_1 -2\gamma L -2|\theta_2|+\theta_2 -2|\theta_2|\theta_1 +2\theta_2\theta_1 >0.
	\end{equation*}
	It follows that
	\begin{align*}
		\lim\limits_{k\to \infty} q_k=&\lim\limits_{k\to \infty} \Big(1-2\theta_1-2\gamma L+\theta_2+2\theta_1\theta_2 -|\theta_2|-\theta_2^2-\theta_1^2\\
	&+(1-\lambda_k)(\theta_1^2-\theta_1-\theta_1|\theta_2|)+(1-\lambda_k)(\theta_2^2-\theta_2-\theta_1|\theta_2|) \Big)\\
	=& 1-3\theta_1 -2\gamma L -2|\theta_2|+\theta_2 -2|\theta_2|\theta_1 +2\theta_2\theta_1 >0.
	\end{align*}
	Hence, there exists $k_2\geq k_0$ such that $q_k>0$ for all $k\geq k_2$. As a result, 
	\begin{align}\label{eq3.27} 
		s_{k+1}\leq (1-\lambda_k)s_k+\lambda_ka_k, \forall k\geq k_2.
	\end{align}
	Let $s_{k_i}$ be the sequence satisfying that $\lim\limits_{i\to \infty} (s_{k_i+1}-s_{k_i})\geq 0.$ 	
	Then \eqref{eq3.26} implies that 
	\begin{align*}
		\limsup_{n\to\infty}& \Big((1-\lambda_k)q_{k_i}\|u_{k_i-1}-u_{k_i-2}\|^2\Big)\\
		\leq & \limsup_{i\to \infty} \big((s_{k_i}-s_{k_i+1})+\lambda_{k_i}(q_{k_i}-s_{k_i})\big)\\
		\leq & -\liminf_{i\to \infty} (\big(s_{k_i+1}-s_{k_i})\leq 0.
	\end{align*}
	Because $\lim\limits_{i\to\infty}(1-\lambda_{k_i})q_{k_i}>0$, it holds that 
	\begin{equation}\label{eq3.28} 
		\lim\limits_{i\to\infty} \|u_{k_i-1}-u_{k_i-2}\|=0=\lim\limits_{i\to\infty} \|u_{k_i+1}-u_{k_i}\|.
	\end{equation}
	So 
	\begin{align}\label{eq3.29} 
		\lim\limits_{i\to \infty} \|w_{k_i}-u_{k_i}\|=\lim\limits_{i\to\infty} \|\theta_1 (u_{k_i}-u_{k_i-1})+\theta_2(u_{k_i-1}-u_{k_i-2})\|=0.
	\end{align}
	By \eqref{eq3.28} and \eqref{eq3.29} we get 
	\begin{align}\label{eq3.30} 
		\lim\limits_{i\to\infty} \|z_{k_i}-w_{k_i}\|=\lim\limits_{i\to\infty} \lambda_{k_i}\|w^*-w_{k_i}\|=0.
	\end{align}
	Because $\lim\limits_{k\to\infty} \lambda_k =0$ it follows that 
	\begin{align}\label{eq3.31} 
		\lim\limits_{i\to\infty} \|z_{k_i}-w_{k_i}\|=\lim\limits_{i\to\infty} \lambda_{k_i}\|w^*-w_{k_i}\|=0.
	\end{align}
	By \eqref{eq3.30} and \eqref{eq3.31} one has 
	\begin{equation}\label{eq3.32} 
		\lim\limits_{i\to\infty} \|z_{k_i}-u_{k_i+1}\|=0.
	\end{equation}
	Because of \eqref{eq3.28} and the Lipschit continuity of $G$, we get
	\begin{equation}\label{eq3.33} 
		\lim\limits_{i\to\infty} \|Gu_{k_i+1}-Gu_{k_i}\|=0.
	\end{equation}
	Due to Lemma \ref{lm3.4}, the sequence $\{u_{k_i}\}$ is bounded. Hence,  there exists a subsequence $\{u_{k_{i_j}}\} $ of $\{u_{k_i}\}$ which converges to $\bar{u}\in H$, and
	\begin{align}\label{eq3.34} 
		\limsup\limits_{i\to\infty} \langle w^*-u^*, u_{k_i}-u^*\rangle =\lim\limits\langle w^*-u^*, u_{k_{i_j}}-u^*\rangle =\langle w^*-u^*, \bar{u}-u^*\rangle. 
	\end{align}
	Let $(x, y)\in \gr(F+G)$. Then $\gamma (y-Gx)\in \gamma Fx$. Due to this, \eqref{eq3.3} and the $\mu_F$-monotone of $F$, we have that
	\begin{align*}
		\langle \gamma (y-Gx)-z_{k_{i_j}} +\gamma Gu_{k_{i_j}} +\gamma (1-\lambda_{k_{i_j}})(Gu_{k_{i_j}}-Gu_{k_{i_j}-1})+u_{k_{i_j}+1}, x-u_{k_{i_j}+1}\rangle \geq \gamma \mu_F\|u_{k_{i_j}+1}-x\|^2. 
	\end{align*}
	Moreover, because of the $\mu_G$-monotone of $G$, it holds that 
	\begin{align}\label{eq3.35} 
	&	\langle y, x-u_{k_{i_j}+1}\rangle \geq  \dfrac{1}{\gamma} \langle \gamma Gx+z_{k_{i_j}}-\gamma Gu_{k_{i_j}}-\gamma (1-\lambda_{k_{i_j}})  \notag \\
		& (Gu_{k_{i_j}}-Gu_{k_{i_J}-1})-u_{k_{i_j}+1}, x-u_{k_{i_j}+1}\rangle+\mu_F\|u_{k_{i_j}}-x\|^2 \notag \\
		=& \langle Gx-Gu_{k_{i_j}+1}, x-u_{k_{i_j}+1}\rangle +\langle Gu_{k_{i_j}+1}-Gu_{k_{i_j}}, x-u_{k_{i_j}+1}\rangle \notag \\
		& +(1-\lambda_{k_{i_j}})\langle Gu_{k_{i_j}-1}-Gu_{k_{i_j}}, x-u_{k_{i_j}+1}\rangle +\dfrac{1}{\gamma} \langle z_{k_{i_j}}-u_{k_{i_j}+1}, x-u_{k_{i_j}+1}\rangle +\mu_F\|u_{k_{i_j}}-x\|^2 \notag  \\
		& \geq \langle Gu_{k_{i_j}+1}-Gu_{k_{i_j}}, x-u_{k_{i_j}+1}\rangle +(1-\lambda_{k_{i_j}})\langle Gu_{k_{i_j}-1}-Gu_{k_{i_j}}, x-u_{k_{i_j}+1}\rangle \notag \\ 
        &+\dfrac{1}{\gamma} \langle z_{k_{i_j}}-u_{k_{i_j}+1}, x-u_{k_{i_j}+1}\rangle+\mu_G \|x-u_{k_{i_j}+1}\|^2+\mu_F\|u_{k_{i_j}+1}-x\|^2\notag \\
        =& \langle Gu_{k_{i_j}+1}-Gu_{k_{i_j}}, x-u_{k_{i_j}+1}\rangle +(1-\lambda_{k_{i_j}})\langle Gu_{k_{i_j}-1}-Gu_{k_{i_j}}, x-u_{k_{i_j}+1}\rangle \notag \\ &+\dfrac{1}{\gamma} \langle z_{k_{i_j}}-u_{k_{i_j}+1}, x-u_{k_{i_j}+1}\rangle+(\mu_G+\mu_F) \|x-u_{k_{i_j}+1}\|^2\notag \\
        \geq & \langle Gu_{k_{i_j}+1}-Gu_{k_{i_j}}, x-u_{k_{i_j}+1}\rangle +(1-\lambda_{k_{i_j}})\langle Gu_{k_{i_j}-1}-Gu_{k_{i_j}}, x-u_{k_{i_j}+1}\rangle \notag \\ 
        &+\dfrac{1}{\gamma} \langle z_{k_{i_j}}-u_{k_{i_j}+1}, x-u_{k_{i_j}+1}\rangle
	\end{align}
    The last inequality holds by condition \eqref{dk5 gt1}. 
	As $j\to\infty$ in \eqref{eq3.35} with using \eqref{eq3.32} and \eqref{eq3.33} we get 
	\begin{equation*}
		 \langle y, x-\bar{u}\rangle \geq 0. 
	\end{equation*}
	Due to Lemma \ref{lm2.2}, $F+G$ is maximal $ \mu_F+\mu_G$-monotone. Hence $\bar{u}\in \zer(F+G)$ by Lemma \ref{max sum}. 
	\\
Because
	$u^*={\mbox{\rm zer}{(F+G)}} w^*$, \eqref{eq3.34} and the characterization of the metric projection imply that 
	\begin{equation}\label{eq3.36} 
		\limsup\limits_{i\to\infty} w^*-u^*, u_{k_i}-u^*=\langle w^*-u^*, \bar{u}-u^*\rangle \leq 0.
	\end{equation}
By \eqref{eq3.29}, \eqref{eq3.30} and \eqref{eq3.36} we get $\limsup\limits_{i\to\infty} s_{k_i}\leq 0$. Hence by condition $\sum\limits_{k=1}^\infty \lambda_k=0$, Lemma \ref{lm2.1} and \eqref{eq3.27} we obtain that
\begin{equation*} 
	\lim\limits_{k\to\infty} s_{k}=0.
	\end{equation*} 
	Because of this fact and \eqref{eq3.22} we conclude that $\{u_k\}$ converges to $u^*={\mbox{\rm zer}{(F+G)}} w^*$, as claimed. 
\end{proof}

	\section{Applications}\label{sec5}
	In this section, we treat problems such as constrained optimization problems (COPs), mixed variational inequalities (MVIs), and variational inequalities (VIs) as special cases of problem \eqref{inc1}. Based on this formulation, we derive a two-step inertial forward-reflected-anchored-backward splitting algorithm for these problems. Under appropriate assumptions on the operators, the sequences generated by the corresponding algorithms are shown to converge strongly. 
	\subsection{Application to COP }
	Consider the constrained optimization problem (COP)
	\begin{equation}\label{cop1}
		\min_{u\in H} f(u)+g(u)
	\end{equation}
	where \(g: H\longrightarrow \R\) is continuous differential and convex, and \(f: H\longrightarrow \R\) is a proper, l.s.c. convex real value function. Note that   \(f\) may lack differentiability, and when \(f\equiv 0\),   problem \eqref{cop1} simplifies to an unconstrained optimization problem. As mentioned earlier, this problem can be reformulated as an inclusion problem \eqref{inc1} with  $F=\partial f$ and $G=\nabla g$. 	Then  \eqref{eqalg} will be reduced to the following 
    \begin{align}\label{eqalg2}
          u_{k+1}=&\mbox{prox}_{\gamma f} (\lambda_k w^*+(1-\lambda_k)\big(u_k-\theta_1(u_k-u_{k-1})+\theta_2(u_{k-1}-u_{k-2}) \notag \\
    & -\gamma \nabla g u_k -\gamma (1-\lambda_k)(\nabla g u_k-\nabla g u_{k-1})\big), \quad \forall k\geq 1. 
    \end{align}

We make the following assumption on the function $g$.
\begin{assumption}
    The function $g$ is convex and $\nabla g$ is Lipschitz continuous with constant $L>0$, and $\gamma \in \left(0, \dfrac{1}{2L}\right)$.
\end{assumption}
	
	Observe that when $g$ is convex,   $\nabla g$ is monotone. In addition, \(F=\partial f\) is maximal monotone \cite{ROCK}.   Therefore, the conditions related to the operators in Assumption \ref{gt1} hold. Consequently, if the parameters $\theta_1, \theta_2$ are chosen to satisfy condition \eqref{dk4 gt1},  the sequence generated by \eqref{eqalg2} strongly converges to the solution of the COP.

	\subsection{Application to MVIP}
	Now we examine the mixed variational inequality problem (MVIP):
	\begin{equation}\label{mvip1}
		\mbox{Find \(u^*\in H\) such that \(\langle T(u^*), u-u^*\rangle +f(u^*)-f(u)\geq 0\) for all $u\in H$,}
	\end{equation}
	where $T: H\longrightarrow H$ is a vector-valued operator and $f: H\longrightarrow \R$ is a proper, l.s.c. convex function. Once again, this problem can be rewritten as an inclusion problem of the form \eqref{inc1} with $G=T$ and $F = \partial f$. Observe that when $F = \partial f$ we have   
	$$J_{\gamma F}(u)=\mbox{prox}_{\gamma f}(u).$$
    Hence \eqref{eqalg} reduces to the following 
      \begin{align}\label{eqalg3}
          u_{k+1}=&\mbox{prox}_{\gamma f} (\lambda_k w^*+(1-\lambda_k)\big(u_k-\theta_1(u_k-u_{k-1})+\theta_2(u_{k-1}-u_{k-2}) \notag \\
    & -\gamma T u_k -\gamma (1-\lambda_k)(T u_k-T u_{k-1})\big), \quad \forall k\geq 1. 
    \end{align} 
	In the case where $f$ is convex,  $T$ is  monotone and $L$-Lipschitz  continuous,  and $\gamma\in (0, \frac{1}{2L})$, then all conditions related to the operators in Assumption \ref{gt1}  hold. As a result, the sequence generated by \eqref{eqalg3} converges strongly to the solution of problem \ref{mvip1} with appropriate values of the parameters. 
	
\subsection{Application to VIP}
	We now consider a special case of problem \ref{mvip1} where $f=0$. This problem is known as the variational inequality problem and is stated as follows.	
	\begin{equation}\label{vip1}
		\mbox{Find \(u^*\in C\) such that \(\langle T(u^*), u-u^*\rangle  \geq 0, \quad \forall u\in C\),}
	\end{equation} 
 where \(C\) is a closed convex subset of \(H\), $F: C\longrightarrow H$ is an operator. 	We denote this problem by $\mbox{VIP}(F, C)$. 
	
The variational inequality problems (VIPs) in \eqref{vip1} can be reformulated as an inclusion problem of the form  $0\in (F+G)(x)$ with $G=T$ and $F=N_C$. Again, in this case, we have that 
 $$J_{\gamma F}(u)  =P_C(u) .$$	Hence the algorithm \eqref{eqalg} reads as. 
	 \begin{align}\label{eqalg4}
          u_{k+1}=&P_C  (\lambda_k w^*+(1-\lambda_k)\big(u_k-\theta_1(u_k-u_{k-1})+\theta_2(u_{k-1}-u_{k-2}) \notag \\
    & -\gamma T u_k -\gamma (1-\lambda_k)(T u_k-T u_{k-1})\big), \quad \forall k\geq 1. 
    \end{align} 
	 When the operator \(T\) satisfies the following assumption: 
     \begin{assumption}
         $T$ is monotone and Lipschitz continuous with constant \(L>0\),
     \end{assumption}
\noindent Then, all conditions in Assumption \ref{gt1} related to the operators also hold. Once again, the sequence generated by \eqref{eqalg4} converges strongly to the solution of problem \eqref{vip1}.

\section{conclusion} \label{sec6}

We have established the strong convergence of a sequence generated by the two-step inertial forward–reflected–anchored–backward splitting algorithm for solving the non-monotone inclusion problem \eqref{inc1} in a real Hilbert space. Additionally, we provide a characterization of when a generalized monotone operator becomes maximal. Moreover, we present conditions that ensure the maximality of the sum of two generalized monotone operators. Finally, we discuss applications to related problems, including constrained optimization, mixed variational inequality problems, and variational problems.



\begin{thebibliography}{99}
	
	
	
	
	
	\bibitem{ABBAS}
	B. Abbas and H. Attouch,  Dynamical systems and forward-backward
	algorithms associated with the sum of a convex subdifferential and
	a monotone cocoercive operator, Optimization. 64, no. 10 (2015), 2223--2252.
	
	
	
	\bibitem{AVR}
	M. Avriel, W.E. Diewert, S. Schaible, and I. Zang, Generalized Concavity,  Society for Industrial
	and Applied Mathematics, 2010.
	
	\bibitem{BauschkeCombettes}
	H.	Bauschke, and P. Combettes, Convex Analysis and Monotone Operator Theory in Hilbert spaces, CMS Books in Mathematics, Springer, 2011.
	

\bibitem{BING}	T. Bing, S.Y. Cho, Strong convergence of inertial forward-backward methods for solving monotone
inclusions. Appl Anal. https://doi.org/10.1080/00036811.2021.1892080,  (2021)

	
	\bibitem{BlumOettli94} E. Blum and W. Oettli, From optimization and variational inequalities to equilibrium problems, Math. Student. 63 (1994), 123-145.
	
	\bibitem{BC_SICON}
	R.I. Bo\c t, and E.R. Csetnek, Second order forward-backward dynamical systems for monotone inclusion problems, SIAM Journal on Control and Optimization. 54 (2016), 1423-1443.
	

	
	\bibitem{22}
	R.I. Bo\c t, E.R. Csetnek, and P.T. Vuong, The forward–backward–
	forward method from continuous and discrete perspective for pseudomonotone
	variational inequalities in Hilbert spaces, Eur. J. Oper. Res.
	287 no. 1 (2020),  49--60.
	
	
	
	\bibitem{BCJMAA} R.I. Bo\c t, and E.R. Csetnek, Convergence rates for forward-backward dynamical systems associated with strongly monotone inclusions, Journal of Mathematical Analysis and Applications. 457 (2018), 1135--1152.
	

	
	\bibitem{BSV} R.I. Bo\c t, M. Sedlmayer, and P.T. Vuong,  A Relaxed Inertial Forward-Backward-Forward Algorithm for Solving Monotone Inclusions with Application to GANs, (2020) a  G(x)iv, 2003.07886.
	
	\bibitem{Cav} E. Cavazzuti, P. Pappalardo, and M. Passacantando, Nash Equilibria, Variational Inequalities, and Dynamical Systems, J. Optim. Theory.  114 (2002), 491-506.
	
	
	
	\bibitem{CHIN} I. Chinedu, A. Maggie, O.A. Kazeem,   Two-step inertial forward–reflected–anchored–backward
splitting algorithm for solving monotone inclusion problems, Computational and Applied Mathematics 42:351  (2023) https://doi.org/10.1007/s40314-023-02485-6

	
	\bibitem{CSE}
	E.R.~Csetnek, Y. Malitsky, and M.K. Tam, Shadow Douglas–Rachford Splitting for Monotone Inclusions,
	Appl. Math Optim. 80 (2019), 665–678.
	
	\bibitem{Minh}M.N. Dao, and H.N. Phan,	Adaptive Douglas-Rashford splitting algorithm for the sum of two operators, Siam J. Optim.  29 no. 4 (2019) 2697--2724.
	
	
	
		\bibitem{Liang}  J. Liang, Convergence Rates of First-Order Operator Splitting Methods, Ph.D. thesis, Normandie University, France, GREYC CNRS UMR 6072, (2016).

\bibitem{LEM} B. Lemaire:  Which fixed point does the iteration method select? Recent advances in optimization. Spring
Berlin, Germany 452, (1997) 154--157.  
    
	\bibitem{MER}
	P.L. Lions, B. Mercier, Splitting algorithms for the sum of two nonlinear operators, SIAM J. Numer. Anal.  16 (1979)
	964–979.
	
	\bibitem{MAL} Y. Malitsky, M.K. Tam, A Forward-Backward splitting method for monotone inclusions without cocoercivity,  (2020),	a  G(x)iv , a  G(x)iv:1808.04162.
	
	
\bibitem{MAI} P.E. Maing\'e:  Approximation methods for common fixed points of nonexpansive mappings in Hilbert
spaces. J Math Anal Appl 325, no. 1, (2007) 469–479. 
	
	\bibitem{PARS} S. Parsegov, A. Polyakov, and P. Shcherbakov, Nonlinear fixed-time
	control protocol for uniform allocation of agents on a segment, in Proc. IEEE 51st IEEE Conf. Decis. Control (CDC). 87 (2013), 133–136.
	
	\bibitem{PAS}
	G.B. Passty,  Ergodic convergence to a zero of the sum of monotone operators in Hilbert spaces, J. Math. Anal. Appl. 72 (1979), 383–390.
	
	
	
	\bibitem{ROCK} R.T. Rockafellar, Monotone operators and the proximal point algorithm, Siam J. Control and optimization. 14 no. 5 (1976), 877--898.
	
\bibitem{SAE}	S. Saejung, P. Yotkaew,  Approximation of zeros of inverse strongly monotone operators in Banach
spaces. Nonlinear Anal., 75, (2012) 742–750.

	\bibitem{TAK} S. Takahashi, W. Takahashi, M. Toyoda, Strong convergence theorems for maximal monotone operators
with nonlinear mappings in Hilbert spaces. J Optim Theory Appl., 147 ( 2010) 27–41.  

	\bibitem{TSENG} P.A. Tseng, A Modified forward-backward splitting method for maximal monotone mappings. SIAM J.Control Optim.,  38,( 2000) 431–446.

	
	



	
\bibitem{WANG2018} Y.	Wang, F. Wang:  Strong convergence of the forward-backward splitting method with multiple parameters
in Hilbert spaces. Optimization, 67, (2018) 493–505. 
	
\end{thebibliography}
\end{document}